\documentclass[11pt,a4paper]{amsart}

\usepackage{amsmath,amssymb,amsthm,amscd,verbatim, enumerate}
\usepackage{graphicx}
\usepackage{epsfig}
\usepackage{hyperref}

\newtheorem{thm}{Theorem}

\newtheorem{claim}[thm]{Claim}
\newtheorem{conj}[thm]{Conjecture}

\newtheorem{cor}[thm]{Corollary}
\newtheorem*{thm*}{Theorem}

\theoremstyle{definition}

\theoremstyle{remark}

\newcommand{\NOS}{\mathbb{N}}
\newcommand{\OS}{\mathbb{S}}

\def\acc#1{\left\{ #1\right\}}

\def\floor#1{\left\lfloor #1\right\rfloor}
\def\ceil#1{\left\lceil #1\right\rceil}

\renewcommand{\le}{\leqslant}

\renewcommand{\ge}{\geqslant}

\begin{document}

\title{Islands in graphs on surfaces}

\author{Louis Esperet} \address{Laboratoire G-SCOP (CNRS,
   Grenoble-INP), Grenoble, France}
\email{louis.esperet@g-scop.fr}

\author{Pascal Ochem} \address{LIRMM (CNRS, Universit\'e de
  Montpellier), Montpellier, France }
   
\email{ochem@lirmm.fr}

\thanks{Louis Esperet is partially supported by ANR Project Heredia
  (\textsc{anr-10-jcjc-0204-01}), ANR Project Stint
  (\textsc{anr-13-bs02-0007}), and LabEx PERSYVAL-Lab
  (\textsc{anr-11-labx-0025}).}

\date{}
\sloppy

\begin{abstract}
An island in a graph is a set $X$ of vertices, such that each element
of $X$ has few neighbors outside $X$. In this paper, we prove several
bounds on the size of islands in large graphs embeddable on fixed
surfaces. As direct consequences of our results, we obtain that:
\begin{enumerate}
\item Every graph of genus $g$ can be colored from lists of size 5, in
  such a way that each monochromatic component has size
  $O(g)$. Moreover all but $O(g)$ vertices lie in monochromatic
  components of size at most 3.
\item Every triangle-free graph of genus $g$ can be colored from lists
  of size 3, in such a way that each monochromatic component has
  size $O(g)$. Moreover all but $O(g)$ vertices lie in
  monochromatic components of size at most 10.
\item Every graph of girth at least 6 and genus $g$ can be colored
  from lists of size 2, in such a way that each monochromatic
  component has size $O(g)$. Moreover all but $O(g)$ vertices lie in
  monochromatic components of size at most 16.
\end{enumerate}
While (2) is optimal up to the size of the components, we conjecture
that the size of the lists can be decreased to 4 in (1), and the girth
can be decreased to 5 in (3). We also study the complexity of
minimizing the size of monochromatic components in 2-colorings of planar
graphs.
\end{abstract}

\maketitle

\section{Introduction}

In this paper we consider a relaxed version of the classical notion of
proper coloring of a graph. We are interested in vertex colorings of
graphs with the property that each color class consists of the
disjoint union of (connected) components of bounded size. These
components are said to be \emph{monochromatic}, and the \emph{size} of
a monochromatic component is its number of vertices. A proper coloring
is the same as a coloring in which every monochromatic component has
size 1, so by allowing monochromatic components of larger size, one
expects that the minimum number of colors needed might decrease
significantly. For instance it was proved by Haxell, Szab\'o and
Tardos~\cite{HST03} that every graph with maximum degree at most $5$
can be 2-colored in such a way that all monochromatic components have
size at most 20000 (such graphs have chromatic number as large as 6).

\smallskip

It was conjectured by Hadwiger that every graph with no $K_{t}$-minor
has a proper coloring with $t-1$ colors. The case $t=5$ was shown to
be equivalent to the famous 4 Color Theorem, which states that every
planar graph has a proper 4-coloring. On the other hand, it was proved
by Kleinberg, Motwani, Raghavan, and Venkatasubramanian~\cite{KMRV97},
and independently by Alon, Ding, Oporowski and Vertigan~\cite{ADOV03}
that there is no constant $c$ such that every planar graph has a
3-coloring in which every monochromatic component has size at most $c$. More
generally, for every $t$, there are graphs with no $K_{t}$-minor that
cannot be colored with $t-2$ colors such that all monochromatic
components have size bounded by a function of $t$ (see~\cite{LMST08}). It follows that the
bound predicted by Hadwiger's conjecture (and proved for $t=5,6$) on
the chromatic number of a graph with no $K_t$-minor is best possible,
even in our relaxed setting.

\smallskip

In this paper we prove that the bound can be significantly decreased
in the specific case of graphs embeddable on surfaces of bounded
genus. Given a graph $G$, a $k$-list assignment $L$ (for the vertices
of $G$) is a collection of lists $L(v)$, $v\in V(G)$, such that each
list contains at least $k$ elements. Given a list assignment $L$, an
$L$-coloring $c$ of $G$ is the choice of an element $c(v)\in L(v)$ for
each vertex $v\in V(G)$. Unless stated otherwise, such a coloring is
not necessarily proper. It was proved by Thomassen~\cite{T94} that for
every planar graph $G$ and every $5$-list assignment $L$, the graph
$G$ has a proper $L$-coloring. We will prove that the same holds from
any graph embeddable on a surface of genus $g$, provided that
monochromatic components are only required to have size bouned by
$O(g)$ (Theorem~\ref{th:g3col}). The fact that cliques of order
$\Omega(\sqrt{g})$ can be embedded on such surfaces shows that the
size of monochromatic components has to depend on $g$. Moreover we
will show that we can find a list-coloring in which all vertices
except $O(g)$ of them lie in monochromatic components of size at most
3.

\smallskip

A theorem of Gr\"{o}tzsch~\cite{G59} states that every triangle-free
planar graph has a proper 3-coloring. Esperet and Joret~\cite{EJ13}
proved that there exist no constant $c$ such that every triangle-free
planar graph has a 2-coloring in which every monochromatic component has size at
most $c$. Hence, it follows again that Gr\"{o}tzsch's theorem cannot
be improved even in our relaxed setting. We will show however that
every triangle-free graph embeddable on a surface of genus $g$ can be
colored from any 3-list assignment, in such a way that all
monochromatic components have size $O(g)$, and all vertices except
$O(g)$ lie in a monochromatic component of size at most 10
(Theorem~\ref{th:g4col}). The case of triangle-free graph is
particularly interesting because Voigt~\cite{V95} proved that there
exists a triangle-free planar graph $G$ and a 3-list assignment $L$ such that
$G$ is not $L$-colorable. So our result is non-trivial (and previously
unknown, as far as we are aware of) even in the case of planar graphs.

\smallskip

The \emph{girth} of a graph $G$ is the smallest size of a cycle in
$G$. We will also show that every graph of girth at least 6 embeddable
on a surface of genus $g$ can be colored from any 2-list assignment,
in such a way that all monochromatic components have size $O(g)$, and
all vertices except $O(g)$ lie in a monochromatic component of size at
most 16 (Theorem~\ref{th:g6col}).

\smallskip

All these results are direct consequences of purely structural results
on (large) graphs embeddable on surfaces of bounded genus.  Given a
graph $G$, a \emph{$k$-island} of $G$ is a non-empty set $X$ of
vertices of $G$ such that each vertex of $X$ has at most $k$ neighbors
outside $X$ in $G$. The \emph{size} of a $k$-island is the number
of vertices it contains. In Section~\ref{sec:island} we show that

\begin{enumerate}
\item large graphs of bounded genus have a 4-island
of size at most 3 (Theorem~\ref{th:g3island});
\item large triangle-free graphs of bounded genus have a 2-island
of size at most 10 (Theorem~\ref{th:g4island});
\item large graphs of girth at least 6 and bounded genus have a 1-island
of size at most 16 (Theorem~\ref{th:g6island}).
\end{enumerate}

The proofs of these three results use the discharging method and are
very similar, but unfortunately each one has some particularities and
therefore we have not been able to factorize them.

In Section~\ref{sec:complex}, we study the computational aspects of
minimizing the size of monochromatic components in 2-colorings of
graphs. We show that approximating this minimum within a constant
multiplicative factor is NP-hard, even when the input graph is a
2-degenerate graph of girth at least 8, or a 2-degenerate
triangle-free planar graph.

\section{Graphs on surfaces}

All the graphs in this paper are simple (i.e., without loops and
multiple edges).

In this paper, a {\em surface} is a non-null compact connected
2-manifold without boundary.  We refer the reader to the monograph of Mohar
and Thomassen~\cite{MoTh} for background on graphs on surfaces.

A surface can be orientable or non-orientable. The \emph{orientable
  surface~$\OS_h$ of genus~$h$} is obtained by adding $h\ge0$
\emph{handles} to the sphere; while the \emph{non-orientable
  surface~$\NOS_k$ of genus~$k$} is formed by adding $k\ge1$
\emph{cross-caps} to the sphere. The \emph{Euler
  characteristic~$\chi(\Sigma)$} of a surface~$\Sigma$ is $2-2h$ if
$\Sigma=\OS_h$, and $2-k$ if $\Sigma=\NOS_k$.

We say that an embedding is \emph{cellular} if every face is
homeomorphic to an open disc of~$\mathbb{R}^2$. Euler's Formula states
that if~$G$ is a graph with a cellular embedding in~$\Sigma$, with
vertex-set~$V$, edge-set~$E$ and face-set~$F$, then
$|V|-|E|+|F|\:=\:\chi(\Sigma)$.

Finally, if $f$ is a face of a graph~$G$ cellularly embedded in a
surface~$\Sigma$, then a \emph{boundary walk of~$f$} is a walk
consisting of vertices and edges as they are encountered when walking
along the whole boundary of~$f$, starting at some vertex and following
some orientation of the face. The \emph{degree of a face~$f$},
denoted~$d(f)$, is the number of edges on a boundary walk of~$f$ (note
that some edges may be counted more than once).

\section{Islands in graphs on surfaces}\label{sec:island}

Recall that a \emph{$k$-island} in a graph $G$ is a non-empty set $X$
of vertices of $G$ such that each vertex of $X$ has at most $k$
neighbors outside $X$ in $G$, and that the \emph{size} of $X$ is its
cardinality $|X|$.

\begin{thm}\label{th:g3island}
Let $\chi$ be an integer, and let $G$ be a connected graph that can be
embedded on a surface of Euler characteristic $\chi$. If $G$ has more
than $-72 \chi$ vertices, then it contains a 4-island of size at most
3.
\end{thm}

\begin{proof}
For the sake of contradiction, we assume that there exists a $\chi$
and a connected graph $G$ that can be embedded on a surface of Euler
characteristic $\chi$, and with more than $-72 \chi$ vertices, but
without any 4-island of size at most 3. We choose such a graph $G$ in
such way that the integer $\chi$ is maximal. By maximality of $\chi$,
$G$ has no embedding on a surface with higher Euler characteristic,
and then using~\cite[Propositions 3.4.1 and 3.4.2]{MoTh} we can assume
that $G$ has a cellular embedding in $\Sigma$ (in the non-orientable
case we use the fact that $G$ is not a tree, which easily follows from the fact
that $G$ has no 4-island of size at most 3). In the remainder, by a
slight abuse of notation we identify $G$ with its embedding in
$\Sigma$. 

\smallskip

We can assume that the embedding of $G$ in $\Sigma$ is edge-maximal
(with respect to $G$ being a simple graph), since if a graph
obtained from $G$ by adding an edge contains a 4-island of size at
most 3, then so does $G$. In particular, we can assume that for every
vertex $v$ of $G$, there is a circular order on the neighbors of $v$
such that any two consecutive vertices in the order are adjacent in
$G$ (note that it does not necessarily mean that $G$ triangulates
$\Sigma$, since the edge between two consecutive neighbors of $v$
might not lie in a face containing $v$).

\smallskip

Since $G$ does not contain any 4-island of size at most 3, {\bf (1)}
$G$ has mininum degree at least 5; {\bf (2)} $G$ does not contain
any path of at most $3$ vertices of degree at most 6 in which the two
end-vertices have degree 5, and {\bf (3)} $G$ does not contain
any triangle whose vertices all have degree at most 6.

\smallskip

We now use the classical discharging method. First, every vertex $v$
of $G$ is assigned a charge $\rho(v)=d(v)-6$ (since $G$ is simple, all
faces have degree at least 3 and then by Euler's Formula, the sum of
the charge on all vertices is at most $-6 \chi$). Then, we locally
move the charge as follows: {\bf (R1)} Every vertex of degree at least
7 gives a charge of $\tfrac{1}{4}$ to every neighbor of degree 5,
{\bf (R2)} and also a charge of $\tfrac{1}{12}$ to every
neighbor of degree 6. {\bf (R3)} Every vertex of degree 6 gives a
charge of $\tfrac{1}{6}$ to every neighbor of degree 5.

\smallskip

We now prove that after the discharging phase, the charge of each
vertex is at least $\tfrac{1}{12}$.

\smallskip

Let $v$ be any vertex of degree 5 (recall that by {\bf (1)} $G$ has
mininum degree at least 5). By {\bf (2)}, each neighbor of $v$ has
degree at least 6, and by {\bf (3)} no two consecutive neighbors of
$v$ have both degree 6. It follows from rules {\bf (R1)} and {\bf
  (R3)} that $v$ receives a charge of at least $3\cdot
\tfrac{1}{4}+2 \cdot \tfrac{1}{6}=\tfrac{13}{12}$. The initial
charge of $v$ was $\rho(v)=-1$, so the new charge is
$\rho'(v)\ge -1+\tfrac{13}{12}=\tfrac{1}{12}$.

Let $v$ be a vertex of degree 6. By {\bf (3)}, no two consecutive
neighbors of $v$ have degree at most 6, so $v$ has at least 3
neighbors of degree at least 7 (from which it receives a charge of at
least $3\cdot \tfrac{1}{12}=\tfrac{1}{4}$ by rule {\bf (R2)}). On the
other hand, by {\bf (2)}, $v$ has at most one neighbor of degree 5, to
which it gives at charge of at most $\tfrac{1}{6}$ by rule {\bf
  (R3)}. The initial charge of $v$ was $\rho(v)=6-6=0$, so the new
charge is $\rho'(v)\ge \tfrac{1}{4}-\tfrac{1}{6}=\tfrac{1}{12}$.

Let $v$ be a vertex of degree 7. Assume first that $v$ has at most two
neighbors of degree 5. Then in this case $v$ gives a charge of at most
$2\cdot \tfrac{1}{4}+5 \cdot \tfrac{1}{12}=\tfrac{11}{12}$ by {\bf
  (R1)} and {\bf (R2)}. Assume now that $v$ has at least three
neighbors of degree 5, and let $x_1,x_2,\ldots,x_7$ be the neighbors
of $v$, in their circular order. Then, by {\bf (2)}, $v$ has precisely 3
neighbors of degree 5, say $x_1,x_3,x_5$ without loss of
generality. Moreover, it follows from {\bf (2)} again that $x_2,x_4$
have degree at least 7. Therefore, in this case $v$ gives a charge of
at most $3\cdot \tfrac{1}{4}+2 \cdot \tfrac{1}{12}=\tfrac{11}{12}$ by
{\bf (R1)} and {\bf (R2)}. In both cases, the new charge of $v$ is
$\rho'(v)\ge 7-6-\tfrac{11}{12}=\tfrac{1}{12}$.

Assume now that $v$ has degree $d\ge 8$. By {\bf (2)}, no two
consecutive neighbors have degree 5, so $v$ gives a charge of at
most $\tfrac{d}2\cdot \tfrac{1}{4}+\tfrac{d}2\cdot \tfrac{1}{12}=
\tfrac{d}{6}$ by {\bf (R1)} and {\bf
  (R2)}. Therefore, the new charge of $v$ is 
$\rho'(v)\ge d-6- \tfrac{d}{6}=d\cdot \tfrac{5}{6}-6\ge \tfrac{2}{3}$
whenever $d\ge 8$.

It follows that the charge of each vertex is
at least $\tfrac{1}{12}$. Therefore, $-6 \chi \ge \tfrac{n}{12}$ and
so $n\le -72 \chi$, contradicting our initial assumption that
$n> -72 \chi$.
\end{proof}

It was pointed out to us by a referee that Theorem~\ref{th:g3island}
is close from a result of Jendrol' and Voss~\cite{JV05}, who proved
that if $G$ has a cellular embedding in a surface of Euler
characteristic $\chi$, and has more than $-83 \chi$ vertices, then it
contains a vertex of degree at most 4, or a triangular face $f$ such
that the sum of the degrees of the vertices on $f$ is at most
18. Because of triangular faces with vertices of degree 5, 6 and 7
respectively, it seems that Theorem~\ref{th:g3island} and this result
are incomparable. More results on light subgraphs in graphs on
surfaces can be found in two surveys of Jendrol' and Voss~\cite{JV02,JV13}

\smallskip

Note that there exist planar graphs with minimum degree 5 in which the
degree 5 vertices are arbitrarily far apart. This shows that our bound
on the size of 4-islands is best possible, even in the case of planar
graphs. Theorem~\ref{th:g3island} has the following direct consequence.

\begin{thm}\label{th:g3col}
For any integer $\chi$, for any graph $G$ that can be embedded on a
surface of Euler characteristic $\chi$, and any 5-list assignment $L$,
$G$ has an $L$-coloring in which every monochromatic component has
size at most $\max(3,-72\chi)$. Moreover, all vertices except at most
$-72\chi$ of them lie in monochromatic components of size at most 3.
\end{thm}

\begin{proof}
Let $G$ be a graph that can be embedded on a surface $\Sigma$ of Euler
characteristic $\chi$, and let $L$ be any 5-list assignment. The proof
proceeds by induction on the number of vertices of $G$. If $G$
contains at most $-72\chi$ vertices, then the theorem is certainly
true. Assume now that $G$ has more than $-72\chi$ vertices. We can
assume that the embedding of $G$ in $\Sigma$ is edge-maximal, since
proving the theorem for a supergraph of $G$ also proves it for $G$. In
particular, we can assume that $G$ is connected, and therefore apply
Theorem~\ref{th:g3island}. It follows that $G$ contains a 4-island $X$
of size at most $3$.Then by the induction hypothesis, the graph
$G\setminus X$ has an $L$-coloring such that each monochromatic
component has size at most $\max(3,-72 \chi)$, and all vertices
except at most $-72\chi$ of them lie in monochromatic components of
size at most 3. We extend this coloring to $G$ by choosing, for each
vertex $v$ of $X$, a color from $L(v)$ that is distinct from that of
its neighbors outside $X$ (if any). The coloring obtained is an
$L$-coloring in which every monochromatic component has size at most
$\max(3,-72 \chi)$. Moreover, all vertices except at most $-72\chi$
of them lie in monochromatic components of size at most 3. This concludes the
proof.
\end{proof}




Cushing and Kierstead~\cite{CK10} proved that for every planar graph
$G$ and every 4-list assignment $L$ to the vertices of $G$, there is
an $L$-coloring of $G$ in which each monochromatic component has size
at most 2. Hence, Theorem~\ref{th:g3col} restricted to planar graphs
is significantly weaker than their result. We conjecture the following:

\begin{conj}\label{conj:g3col}
There is a function $f$ such that for any integer $\chi$, for any
graph $G$ that can be embedded on a surface of Euler characteristic
$\chi$, and any 4-list assignment $L$, $G$ has an $L$-coloring in
which every monochromatic component has size at most $f(\chi)$.
\end{conj}

We believe that any large graph of bounded genus contains a 3-island
of bounded size, which would directly imply
Conjecture~\ref{conj:g3col}, but we have not been able to prove it,
even in the case of planar graphs.

Kawarabayashi and Thomassen~\cite{KT12} proved that every graph that
has an embedding on a surface of Euler characteristic $\chi$ can be
colored with colors $1,2,3,4,5$, in such a way that each color $i\le
4$ is an independent set, while color 5 induces a graph in which each
connected component contains $O(\chi^2)$ vertices. A small variation
in the proof of Theorem~\ref{th:g3col} shows the following corollary.

\begin{cor}
Every graph that has an embedding on a surface of Euler characteristic
$\chi$ can be colored with colors $1,2,3,4,5$, in such a way that each
color $i\le 4$ induces a graph in which each connected components has
size at most 3, while color 5 induces a graph in which each connected
component contains $O(|\chi|)$ vertices.
\end{cor}

We now prove a triangle-free version of Theorem~\ref{th:g3island}.

\begin{thm}\label{th:g4island}
Let $\chi$ be an integer, and let $G$ be a connected triangle-free
graph that can be embedded on a surface of Euler characteristic
$\chi$. If $G$ has more than $-72 \chi$ vertices, then it contains a
2-island of size at most 10.
\end{thm}

\begin{proof}
The proof is similar to that of Theorem~\ref{th:g3island}. We consider
a counterexample $G$ (we can assume that it has a cellular embedding
on some surface of Euler characteristic $\chi$). Since $G$ does not contain any 2-island of size at
most 10, {\bf (1)} $G$ has mininum degree at least 3, and {\bf (2)}
$G$ does not contain any path of at most $10$ vertices of degree at
most 4 in which the two end-vertices have degree 3 (the two
end-vertices are allowed to coincide).

\smallskip

We now use the discharging method. First, every vertex $v$
of $G$ is assigned a charge $\rho(v)=d(v)-4$, and every face $f$ of
$G$ is assigned a charge $\rho(f)=d(f)-4$ (by Euler's Formula, the sum
of the charge on all vertices and faces is equal to $-4 \chi$). Then, we
locally move the charge as described below.

We first choose, for every face $f$ of $G$, an orientation of $f$ and
set it as the positive orientation of $f$ (we do not need to have a
consistent choice of positive orientations, therefore $\Sigma$ is not
required to be orientable). For any face $f$ of $G$, for any
orientation of $f$ (positive or negative), and for any occurrence of a
vertex $v$ of degree 3 in a boundary walk of $f$ according to the
chosen orientation, take a maximal facial walk of $f$ (a walk
consisting only of vertices and edges incident to $f$) starting at $v$
and going around $f$ in the prescribed orientation of $f$, such that
the inner vertices of the walk have degree precisely 4. Let $u$ be the
other end-vertex of the walk. If the walk contains at least $3$ inner
vertices, then the face $f$ gives a charge of $\tfrac16$ to
$v$. Otherwise {\bf (1)}, {\bf (2)} and the maximality of the walk
imply that $u$ has degree at least 5. In this case $u$ gives a charge
of $\tfrac16$ to $v$.

\smallskip

We now prove that after the discharging phase, all vertices and faces
have nonnegative charge.

\smallskip

Let $v$ be any vertex of degree 3 (recall that by {\bf (1)} $G$ has
mininum degree at least 3). Then $v$ appears 6 times in the union
of all boundary walks of faces of $G$ (for each face, we consider a
boundary walk in the positive orientation and a boundary walk in the
negative orientation of the face), and therefore receives 6 times a
charge of $\tfrac16$. The initial charge of $v$ was $\rho(v)=-1$, so
the new charge is $\rho'(v)=-1+6\cdot \tfrac16=0$.

Vertices of degree $4$ start with an initial charge of $4-4=0$ and do
not give or receive any charge. Now let $v$ be a vertex of degree
$d\ge 5$. Consider the facial walks through which it gives a charge of
$\tfrac16$ to some vertices of degree 3, and observe that if a
neighbor $u$ of $v$ is right after $v$ in more than one such facial walk,
then $u$ has degree 3 (and receives exactly $2\cdot\tfrac16=\tfrac13$
from $v$). For if $u$ had degree at least 4 and was just after $v$ in two
facial walks starting at $v$ as defined above, $u$ would have degree
exactly 4 and there would be two paths starting at $u$, each
containing at most $1$ inner vertex (of degree 4) and finishing at a
vertex of degree 3, contradicting {\bf (2)}. It follows that
$\rho'(v)\ge d-4-2d\cdot \tfrac16=\tfrac23 d-4\ge \tfrac{2d}{21}$
whenever $d\ge 7$. If $d=6$, then observe that $v$ is adjacent to at
most three vertices of degree 3 (otherwise $G$ would contain a
2-island of size 5). Therefore, in this case we have $\rho'(v)\ge
2-3\cdot \tfrac13-3\cdot\tfrac16=\tfrac{1}{2}=\tfrac{d}{12}$. If
$d=5$, then the walks through which $v$ gives some charge contain at
most two neighors of $v$, since otherwise $G$ would contain a 2-island
of size at most 10. It follows that in this case we have $\rho'(v)\ge
1-2\cdot \tfrac13=\tfrac{1}{3}=\tfrac{d}{15}$.

Let $f$ be a face of degree $d$ in $G$. If $d=4$, no vertex receives
any charge from $f$, since otherwise $f$ contains a vertex of degree 3
and three vertices of degree 4, and then the vertices of $f$ form a
2-island of size 4. It follows that if $d=4$, $\rho'(f)=0$. Assume now
that $d\ge 5$. For each occurrence of a vertex $v$ of degree 3 that
receives $\tfrac16$ from $f$ in the positive orientation, let $A^+(v)$
be the set consisting of the three vertices of degree exactly 4
following $v$ in the positive orientation of $f$ (the existence of
these vertices follows from the definition of our discharging
procedure). Similary, define $A^-(v)$ for each occurrence of a vertex
$v$ of $f$ receiving some charge from $f$ in the negative orientation
of $f$. Observe that all the sets $A^+(v)$ and $A^-(u)$ are pairwise
disjoint: for a pair of sets $A^+(u)$ and $A^+(v)$, or $A^-(u)$ and
$A^-(v)$, this follows from the definition of these sets and the fact
that they exist only if $u$ and $v$ have degree three. For each pair
$A^+(u),A^-(v)$, if these two sets have non-empty intersection then
$G$ contain a 2-island of size at most 7, which is a contradiction. It
follows that $f$ gives at most $\tfrac16 \cdot\floor{\tfrac{d}3}$.
Since $d\ge 5$ and the face $f$ starts with an initial charge of
$\rho(f)=d-4$, in this case the new charge is $\rho'(f)\ge
d-4-\tfrac16\floor{\tfrac{d}3}\ge \tfrac{d}{6}$.\\

We proved that all vertices and faces have nonnegative charge (if
$G$ is projective planar this is already a contradiction since in this case the total
charge is negative). Moreover, vertices $v$ with degree $d\ge 5$
have a charge $\rho'(v)\ge \tfrac{d}{15}$, while faces $f$ with degree
$d\ge 5$ have a charge $\rho'(v)\ge \tfrac{d}{6}$. We now redistribute
the charge as follows: every vertex of degree at least 5 gives
$\tfrac1{18}$ to every incident face, and then every face $f$ gives
$\tfrac1{54}$ to every occurrence of a vertex of degree 3 or 4 in a
boundary walk of $f$. Each face of degree $d\ge 5$ is left with at
least $\tfrac{d}{6}-\tfrac{d}{54}\ge 0$. Note that a face of degree 4
is incident to at most 3 vertices of degree 3 or 4 (otherwise $G$
would contain a 2-island of size 4), therefore such a face starts with a
charge of 0, receives $\tfrac1{18}$ from a vertex of degree at least
5, and gives at most $3\times \tfrac1{54}=\tfrac1{18}$ to the
remaining vertices of its boundary. Therefore, each face has nonnegative
charge.

Each vertex of degree $d\ge 5$ starts with a charge of at least
$\tfrac{d}{15}$ and gives at most $\tfrac{d}{18}$, thus the remaining
charge is at least $\tfrac{d}{90}\ge\tfrac1{18}$. Each vertex $v$ of degree 3 or 4
starts with a charge of 0 and receives $\tfrac1{54}$ from each
incident face, for a total of at least $\tfrac1{18}$ (note that since
faces give charge to every occurrence of a vertex on their boundary,
this holds even if the number of faces incident to $v$ is less than
$d(v)$ because then $v$ appears several times in a boundary walk of
some face). It follows that the charge of each vertex is
at least $\tfrac1{18}$. Since all faces have nonnegative charge, we have
$-4 \chi \ge \tfrac{n}{18}$ and
so $n\le -72 \chi $, contradicting our initial assumption that
$n>-72 \chi$.
\end{proof}

Our bound on the size of 2-islands is not optimal in the case of
planar graphs: it is possible to show, using a more detailed (and
significantly longer) analysis, that every triangle-free planar graph
contains a 2-island of size at most 5, which is best possible. It is
likely that the result extends to higher surfaces as well, but we
preferred to present a short and simple proof of a slightly weaker
result instead (the most important part of the theorem being that the
island is a 2-island).

Euler's formula shows that every triangle-free planar graph $G$
contains a vertex of degree at most 3. It follows that for any 4-list
assignment $L$, $G$ has a proper $L$-coloring. On the other hand,
Voigt~\cite{V95} proved that there is a triangle-free planar graph
$G$ and a 3-list assignment $L$ such that $G$ is not
$L$-colorable. Using the same proof as that of Theorem~\ref{th:g3col},
Theorem~\ref{th:g4island} has the following direct consequence (which
seems to have been previously unknown even for planar graphs).

\begin{thm}\label{th:g4col}
For any integer $\chi$, for any triangle-free graph $G$ that can be
embedded on a surface of Euler characteristic $\chi$, and any 3-list
assignment $L$, $G$ has an $L$-coloring in which every monochromatic
component has size at most $\max(10,-72\chi)$. Moreover, all vertices
except at most $-72\chi$ of them lie in monochromatic components of
size at most 10.
\end{thm}

Note that the size of the lists in Theorem~\ref{th:g4col} is best
possible: Esperet and Joret~\cite{EJ13} proved that triangle-free
planar graphs $G$ cannot be 2-colored such that each monochromatic
component has bounded size. We conjecture the following:

\begin{conj}\label{conj:g5col}
There is a function $f$ such that for any integer $\chi$, for any
graph $G$ of girth at least 5 that can be embedded on a surface of
Euler characteristic $\chi$, and any 2-list assignment $L$, $G$ has an
$L$-coloring in which every monochromatic component has size at most
$f(\chi)$.
\end{conj}

We now prove a weaker version of this conjecture, for graphs of girth
at least 6 (instead of 5).

\begin{thm}\label{th:g6island}
Let $\chi$ be an integer, and let $G$ be a connected graph of girth at
least six that can be embedded on a surface of Euler characteristic
$\chi$. If $G$ has more than $-357 \chi$ vertices, then it contains a
1-island of size at most 16.
\end{thm}

\begin{proof}
The proof is similar to that of Theorem~\ref{th:g3island}. We consider
a counterexample $G$ (we can assume that it has a cellular
embedding on some surface of Euler characteristic $\chi$). Since $G$
does not contain any 1-island of size at most 16, {\bf (1)} $G$ has
mininum degree at least 2, and {\bf (2)} $G$ does not contain any path
of at most $16$ vertices of degree at most 3 in which the two
end-vertices have degree two (the two end-vertices are allowed to
coincide).

\smallskip

We now use the classical discharging method. First, every vertex $v$
of $G$ is assigned a charge $\rho(v)=2d(v)-6$, and every face $f$ of
$G$ is assigned a charge $\rho(f)=d(f)-6$ (by Euler's Formula, the sum
of the charge on all vertices and faces is equal to $-6 \chi$). Then, we
locally move the charge as described below.

We first choose, for every face $f$ of $G$, an orientation of $f$ and
set it as the positive orientation of $f$. For any face $f$ of $G$,
for any orientation of $f$ (positive or negative), and for any
occurrence of a vertex $v$ of degree two in a boundary walk of $f$
according to the chosen orientation, take a maximal facial walk of $f$
(a walk consisting only of vertices and edges incident to $f$)
starting at $v$ and going around $f$ in the prescribed orientation of
$f$, such that the inner vertices of the walk have degree precisely
3. Let $u$ be the other end-vertex of the walk (note that possibly
$u=v$ if for instance all vertices of $f$ distinct from $v$ have
degree three; another extreme case is that there are no inner vertices
at all and $u$ and $v$ are neighbors). If the walk contains at least
$5$ inner vertices, the face $f$ gives a charge of $\tfrac12$ to
$v$. Otherwise {\bf (1)}, {\bf (2)} and the maximality of the walk
imply that $u$ has degree at least 4. In this case $u$ gives a charge
of $\tfrac12$ to $v$.

\smallskip

We now prove that after the discharging phase, all vertices and faces
have nonnegative charge.

\smallskip

Let $v$ be any vertex of degree two (recall that by {\bf (1)} $G$ has
mininum degree at least 2). Then $v$ appears four times in the union
of all boundary walks of faces of $G$ (for each face, we consider a
boundary walk in the positive orientation and a boundary walk in the
negative orientation of the face), and therefore receives four times a
charge of $\tfrac12$. The initial charge of $v$ was $\rho(v)=-2$, so
the new charge is $\rho'(v)=-2+4\cdot \tfrac12=0$.

Vertices of degree 3 start with an initial charge of $0$, and neither
give nor receive any charge, so after the discharging their charge is
still $0$. Now let $v$ be a vertex of degree $d\ge 4$. Consider the
facial walks through which it gives a charge of $\tfrac12$ to some
vertices of degree two, and observe that if a neighbor $u$ of $v$ is
right after $u$ in more than one such facial walk, then $u$ has degree
two. For if $u$ had degree at least three and was just after $v$
in two facial walks starting at $v$ as defined above, $u$ would have
degree exactly three and there would be two paths starting at $u$,
each containing at most $3$ inner vertices (each of degree 3) and
finishing at a vertex of degree two. Thus $G$ would contain a path on
at most 9 vertices, such that all vertices have degree at most 3 and
the two endpoints have degree two, contradicting {\bf (2)}.

It follows that $v$ gives a charge of at most $d \cdot 2 \cdot
\tfrac12$. Since it starts with an initial charge of $\rho(v)=2d-6$,
its new charge $\rho'(v)$ is at least $2d-6-d=d-6\ge \tfrac{d}7$ as
soon as $d\ge 7$. 

If $d=6$ observe that $v$ cannot be adjacent to at
least 5 vertices of degree 2, since otherwise $v$ together with these
vertices would form a 1-island of size at most 6. Hence if $d=6$, $v$
gives at most $4+2\cdot \tfrac12=5$. Since it starts with an initial
charge $\rho(v)=6$, we have $\rho'(v)\ge 6-5=1\ge \tfrac{d}6$.

If $d=5$, then by the same argument as above it cannot be adjacent to
at least 4 vertices of degree 2. If it is adjacent to at most two vertices of
degree 2, it gives a charge of at most $2+3\cdot \tfrac12=\tfrac72$
and it follows that $\rho'(v)\ge 4-\tfrac72=\tfrac12\ge
\tfrac{d}{10}$.  Otherwise, $v$ is adjacent to exactly three vertices
of degree 2. But then observe that $v$ cannot give a charge of $\tfrac12$
through any of its two neighbors of degree more that two, since
otherwise $G$ would contain a 1-island of size at most 9. Therefore in
this case $v$ gives a charge of 3, and then $\rho'(v)=1\ge \tfrac{d}{5}$.

If $d=4$, then using again the same argument, $v$ cannot be adjacent
to more than two vertices of degree two. Moreover, if $v$ is adjacent
to two vertices of degree two, then it does not give any charge
through its neighbors of degree more than two (in this case it follows
that $\rho'(v)=0$). If $v$ has one neighbor of degree two then it
cannot give charge through more than one neighbor of degree more than
two (otherwise $G$ contains a 1-island of size at most 12), so in this
case we obtain $\rho'(v)\ge 2-1-\tfrac12=\tfrac12\ge \tfrac{d}{8}$. If
$v$ has no neighbor of degree 2, then\footnote{In the planar case we
  can avoid this argument and simply say that in this case $v$ gives
  at most $4\cdot \tfrac12$ and therefore its new charge is at least
  0. This allows to find 1-islands of size at most 12 (instead of 16)
  in any 2-edge-connected planar graph of girth at least 6.} $v$ does
not give charges through more than two of its neighbors, since
otherwise $G$ would contain a 1-island of size at most 16. Thus, in
this case $\rho'(v)\ge 2-2\cdot \tfrac12=1\ge \tfrac{d}{4}$.

We proved that for any vertex $v$, $\rho'(v)\ge 0$, and if $v$ has degree
at least four and is not a vertex of degree four with exactly two
neighbors of degree two, then $\rho'(v)\ge \tfrac1{10} \,d(v)$. 

\smallskip

Let $f$ be a face of degree $d$ in $G$ (since $G$ has girth at least
6, $d\ge 6$). If $d=6$, no vertex receives any charge from $f$, since
otherwise $f$ contains a vertex of degree two and 5 vertices of degree
three, and then the vertices of $f$ form a 1-island of size 6. It
follows that if $d=6$, $\rho'(f)=0$. Assume now that $d\ge 7$. For
each vertex $v$ of degree two that receives $\tfrac12$ from $f$ in the
positive orientation, let $A^+(v)$ be the set consisting of the five
vertices of degree exactly 3 following $v$ in the positive orientation
of $f$ (the existence of these vertices follows from the definition of
our discharging procedure). Similary, define $A^-(v)$ for each vertex
$v$ of $f$ receiving some charge from $f$ in the negative orientation
of $f$. Observe that all the sets $A^+(v)$ and $A^-(u)$ are pairwise
disjoint: for a pair of sets $A^+(u)$ and $A^+(v)$, or $A^-(u)$ and
$A^-(v)$, this follows from the definition of these sets and the fact
that they exist only if $u$ and $v$ have degree two. For each pair
$A^+(u),A^-(v)$, if these two sets have non-empty intersection then
$G$ contain a 1-island of size at most 11, which is a
contradiction. It follows that $f$ gives at most $\tfrac12
\cdot\floor{\tfrac{d}5}$.  Since $d\ge 7$ and the face $f$ starts with
an initial charge of $\rho(f)=d-6$, in this case the new charge is
$\rho'(f)\ge d-6-\tfrac12\floor{\tfrac{d}5}\ge \tfrac{d}{14}$.

\smallskip

Recall that the total charge on the vertices and faces is $-6\chi$,
and we proved that the charge of every vertex and every face is
nonnegative (if $\chi> 0$ this is already a contradiction, since in
this case the total charge is negative). In the previous paragraphs we
also proved that if a vertex or face of degree $d$ has non-zero
charge, then this charge is at least $\tfrac{d}{14}$.

\smallskip

A vertex that has non-zero charge, or is incident to a face of
non-zero charge, or shares a face of degree 6 with a vertex with
non-zero charge, is said to be \emph{heavy}. Observe that every face
$f$ with non-zero charge defines at most $d(f)$ heavy vertices, and
every vertex $v$ with non-zero charge defines at most $4d(v)+1\le
\tfrac{17}4\,d(v)$ heavy vertices. So the number of heavy vertices is
at most $\tfrac{17}4$ times the sum of the degrees of the vertices and
faces with non-zero charge, which by the previous paragraphs is itself
at most 14 times the total charge. Hence, there are at most $14\cdot
\tfrac{17}4\cdot (-6\chi)=-357\chi$ heavy vertices. Since $G$ contains
more than $-357\chi$ vertices, it contains a vertex $v$ that is not
heavy. By the definition of $v$, all the faces incident to $v$ have
degree 6, and all the vertices incident to these faces (including $v$)
have degree 2, 3 or 4 (and if one of these vertices has degree 4, it
has precisely two neighbors of degree 2).

\smallskip

Let $f$ be any face incident to $v$. Since $d(f)=6$, $f$ contains at
least one vertex of degree 4 (since otherwise the vertices of $f$
would form a 1-island of size 6). By definition of $v$, any such
vertex of degree 4 has exactly two neighbors of degree 2. In
particular, two such vertices of degree four cannot be adjacent,
otherwise they would form a 1-island of size at most 6 (together with
their neighbors of degree 2). Thus, we can assume that $f$ contains at
most 3 vertices of degree 4. If each of these vertices has at least
one neighbor of degree 2 outside $f$, then we obtain a 1-island of
size at most 9. It follows that some vertex $u_1$ of degree four on
the boundary of $f$ has its two neighbors of degree two on $f$. Let
$P$ be the set of three vertices of $f$ distinct from $u_1$ and its two
neighbors of degree 2. Then $P$ contains at least one vertex of degree
4, since otherwise we find a 1-island of size 5 in $G$. Using the same
argument as above $P$ contains a vertex $u_2$ of degree 4 such that
its two neighbors of degree 2 belong to $f$. One of these neighbors is
also a neighbor of degree two of $u_1$, since otherwise we have a
1-island consisting of two adjacent vertices of degree two. It follows
that $f$ contains a third vertex $u_3$ of degree 4, having its two
neighbors of degree two on $f$. Therefore, $f$ contains only vertices
of degree 2 and 4, that alternate on its boundary.

Note that the conclusion above holds for any face $f$ incident to
$v$. This implies that that $d(v)\ne 4$, since otherwise $v$ would
have four neighbors of degree two, and $d(v)\ne 2$, since otherwise a
neighbor of degree four of $v$ has at least three neighbors of degree
two.

\smallskip

This final contradiction concludes the proof of the theorem.
\end{proof}

\begin{figure}[htbp]
\begin{center}
\includegraphics[width=10cm]{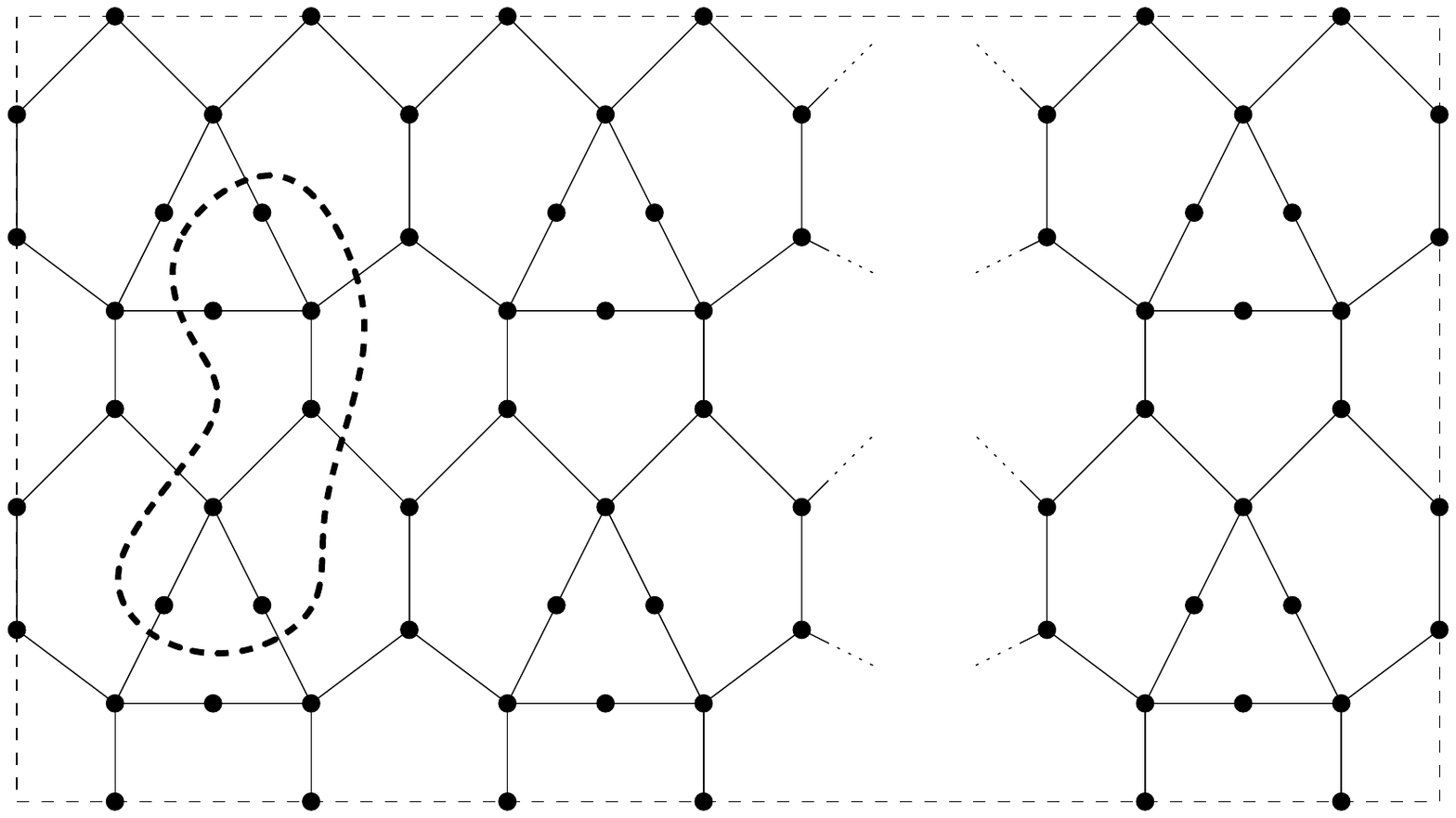}
\caption{Arbitrarily large toroidal graphs of girth 6 with no 1-island of size at most 6. A
  1-island on 7 vertices is highlighted.}
\label{tore7}
\end{center}
\end{figure}

The bound on the size of 1-islands is certainly far from optimal. We
were only able to construct large toroidal graphs of girth 6 with no 1-island of size
at most 6 (see Figure~\ref{tore7}). Using the same proof as that of Theorem~\ref{th:g3col},
the following is a direct consequence of Theorem~\ref{th:g6island}.

\begin{thm}\label{th:g6col}
For any integer $\chi$, for any graph $G$ of girth at least 6 that can
be embedded on a surface of Euler characteristic $\chi$, and any
2-list assignment $L$, $G$ has an $L$-coloring in which every
monochromatic component has size at most
$\max(16,-357\chi)$. Moreover, all vertices except at most $-357\chi$
of them lie in monochromatic components of size at most 16.
\end{thm}

Using the argument
mentioned in the footnote in the proof of Theorem~\ref{th:g6island},
Theorem~\ref{th:g6col} can be slightly improved for planar graphs:

\begin{thm}\label{th:girth6p}
For any planar graph $G$ of girth at least 6 and any 2-list assignment
$L$, $G$ has an $L$-coloring in which every monochromatic component
has size at most $12$.
\end{thm}

Note that it was proved by Borodin, Kostochka, and Yancey~\cite{BKY13}
that every planar graph of girth at least 7 has a 2-coloring in which
every monochromatic component has size at most 2.

\section{Complexity}\label{sec:complex}

In this section we show that it is NP-hard to approximate the minimum
size of the largest monochromatic component in a 2-coloring of a graph
within a constant multiplicative factor. Let us define an
$MC(k)$-coloring as a 2-coloring such that every monochromatic component
has size at most $k$.  Let $\mathcal{MC}(k)$ be the class of graphs having an
$MC(k)$-coloring.

\begin{thm}{\ }
\label{npc}
Let $k\ge 2$ be a fixed integer. The following problems are
NP-complete.

\begin{enumerate}
 \item Given a 2-degenerate graph with girth at least 8 that either is
   in $\mathcal{MC}(2)$ or is not in $\mathcal{MC}(k)$, determine whether it is in
   $\mathcal{MC}(2)$.
 \item Given a 2-degenerate triangle-free planar graph that either is
   in $\mathcal{MC}(k)$ or is not in $\mathcal{MC}(k(k-1))$, determine whether it is in
   $\mathcal{MC}(k)$.
\end{enumerate}
\end{thm}

Before proving Theorem~\ref{npc}, we first describe a gadget used in
the proof and its properties.  Let $t\ge 2$ be an integer.  Let
$T_{x,t}$ be the complete rooted tree of height $3$ with root $x$ such
that every internal node has $5t$ children.  We consider the planar
embedding of $T$ into 4 layers such that the root is on layer 0 and
the leaves are on layer 3. We label the $(5t)^3$ leaves with the
triples in $\{1,2,\dots,5t\}^3$ in lexicographical order from the
leftmost leaf with label $(1,1,1)$ to the rightmost leaf with label
$(5t,5t,5t)$.  Let $J_{y,z,t}$ be the graph obtained from two copies
$T_{y,t}$ and $T_{z,t}$ of $T_{x,t}$ by identifying the leaf labelled
$(l_1,l_2,l_3)$ in $T_{y,t}$ with the leaf labelled $(l_3,l_2,l_1)$ in
$T_{z,t}$, for every triple $(l_3,l_2,l_1)\in \{1,2,\dots,5t\}^3$.

\begin{claim}{\ }
\label{c_j}
\begin{enumerate}[(i)]
 \item The graph $J_{y,z,t}$ is bipartite, 2-degenerate, and the
   distance between $y$ and $z$ is 6.
 \item The girth of $J_{y,z,t}$ is 8.
 \item Every $MC(t)$-coloring of $J_{y,z,t}$ is such that $y$
   and $z$ have the same color.
\end{enumerate}
\end{claim}

\begin{proof}{\ }
\begin{enumerate}[(i)]
 \item Trivial.
 \item Since $J_{y,z,t}$ is bipartite, we suppose for
   contradiction that it contains a cycle $C$ of length 4 or 6.
   Notice that $C$ necessarily contains exactly 2 vertices $u$ and $v$
   of degree 2.  Let $(u_1,u_2,u_3)$ and $(v_1,v_2,v_3)$ be the labels
   of $u$ and $v$ in $T_{y,t}$.  The cycle $C$ consists in a path
   $p_y$ contained in $T_{y,t}$ and a path $p_z$ contained in
   $T_{z,t}$ that both link $u$ to $v$.  Since $u$ and $v$ are
   distinct, there exists an index $i$ such that $u_i\ne v_i$. The
   length of $p_y$ is at least $2(4-i)$ and the length of $p_z$ is at
   least $2i$. Thus, the length of $C$ is at least $2(4-i)+2i=8$, a
   contradiction.
 \item Suppose that $T_{x,t}$ has an $MC(t)$-coloring using
   colors in $\acc{0,1}$ such that the root is colored 0.  Notice that
   a vertex colored $c$ in $T_{x,t}$ has at least $5t-(t-1)=4t+1$
   children with color $1-c$.  This implies that at least $(4t+1)^i$
   vertices in layer $i$ are colored $i \pmod 2$.  Since
   $(4t+1)^3>\tfrac12(5t)^3$, more than half of the leaves are colored
   1.  This forces $y$ and $z$ to have the same color in every
   $MC(t)$-coloring of $J_{y,z,t}$.
\end{enumerate}
\end{proof}

\begin{proof}[Proof of Theorem~\ref{npc}.]
In each case, we make a reduction from 3-uniform hypergraph
2-colorability, which is a well-known NP-complete
problem~\cite{Lov73}.  We consider a 3-uniform hypergraph $H$ and
construct a corresponding graph $G$ as follows.  For every vertex $v$
of $H$, we consider a corresponding vertex $v$ in $G$.  These vertices
are called the \emph{primitive vertices} of $G$.

\smallskip

\noindent (1) We describe the reduction for the first result.  For every
hyperedge $e=\{u_0,u_1,u_2\}$ of $H$, we add a path
$e_1,e_2,\dots,e_{k+1}$ in $G$.  For every vertex $e_j$ in this path,
we take a new copy of $J_{y,z,k}$ and identify the vertex $y$ with $e_j$
and the vertex $z$ with $u_{j \pmod 3}$. By Claim~\ref{c_j}, the girth of
$J_{y,z,k}$ is 8 and thus the girth of $G$ is also 8.

\smallskip

We now show that $G$ is in $\mathcal{MC}(2)$ if $H$ is 2-colorable and that $G$
is not in $\mathcal{MC}(k)$ otherwise.  If $H$ is 2-colorable, then we consider
a 2-coloring of $H$ and color the primitive vertices of $G$
accordingly.  This colors the vertex $z$ of every copy of $J_{y,z,k}$
and we extend this precoloring to all the vertices of $G$ by properly
2-coloring every copy of $J_{y,z,k}$.  All the
monochromatic edges belong to paths $e_1,e_2,\dots,e_{k+1}$
corresponding to hyperedges $e$ in $H$.  In such a path, no three
consecutive vertices can have the same color since it would correspond
to a monochromatic hyperedge in $H$.  This implies that $G$ is in
$\mathcal{MC}(2)$.  

Now suppose for contradiction that $H$ is not 2-colorable and that $G$
is in $\mathcal{MC}(k)$.  Since $H$ is not 2-colorable, any 2-coloring of the
primitive vertices of $G$ is such that there exists three primitive
vertices $u$, $v$, and $w$ in $G$ corresponding to a monochromatic
hyperedge $e=\{u,v,w\}$ in $H$. By Claim~\ref{c_j}, any $MC(k)$-coloring
of the gadgets $J_{y,z,k}$ containing $u$, $v$, or $w$ and extending
the precoloring of $u$, $v$, or $w$, is such that the path
$e_1,e_2,\dots,e_{k+1}$ corresponding to $e$ is monochromatic. This
gives a monochromatic component of size $k+1$, which is a
contradiction.  So, if $H$ is not 2-colorable then $G$ is not in
$\mathcal{MC}(k)$.

\medskip

\noindent (2) The reduction for the second result is similar: for
every hyperedge $e$ in
$H$, we add a path $e_1,e_2,\dots,e_{k(k-1)+1}$ in $G$. Such a path
cannot be monochromatic in an $MC(k(k-1))$-coloring of $G$. We now present
the gadgets that are needed to transfer the color of the primitive
vertices to the paths corresponding to the hyperedges of $H$.

Let $N_{y,z,k}$ be the bipartite graph obtained from two non-adjacent
vertices $y$ and $z$ and a path $v_1,v_2,\dots,v_{3k^4}$ such that $y$
is adjacent to all the vertices $v_{i}$ with $i \equiv 0 \pmod 2$ and
$z$ is adjacent to all the vertices $v_{i}$ with $i \equiv 1 \pmod 2$.
Every $MC(k(k-1))$-coloring of $N_{y,z,k}$ is such that $y$ and $z$
have distinct colors. For if $y$ and $z$ had the same color, say color
0, then at most $2k(k-1)-2$ vertices on the path would be colored 0. Then
the path would contain a monochromatic subpath colored 1 of length at least
$\ceil{\frac{3k^4-(2k(k-1)-2)}{2k(k-1)-1}}\ge k(k-1)+1$.

The gadget $N_{y,z,k}$ can thus be used to force two vertices to have
distinct colors in an $MC(k(k-1))$-coloring. To force two vertices to
have the same color, we could simply chain two copies of $N_{y,z,k}$.
We prefer to use a copy of $K_{2,2k(k-1)-1}$, since it is smaller.

In the last gadget $U_k$ depicted in Figure~\ref{kk1}, the dotted
edges represent copies of $K_{2,2k(k-1)-1}$ and the dashed edges
represent copies of $N_{y,z,k}$. Every vertex $y_i$, $1\le i \le
2(k-1)$, has precisely $k-1$ neighbors connected to $x_N$ by a dotted
edge.\\

The gadget $U_k$ has the following properties:
\begin{enumerate}
 \item Every $MC(k(k-1))$-coloring of $U_k$ is such that $x_N$ and
   $x_S$ have the same color, $x_W$ and $x_E$ have the same color.
 \item There exists an $MC(k)$-coloring of $U_k$ such that $x_N$ and
   $x_W$ have the same color, and there exists an $MC(k)$-coloring of
   $U_k$ such that $x_N$ and $x_W$ have distinct colors.
\end{enumerate}

\begin{figure}[htbp]
\begin{center}
\includegraphics[width=12cm]{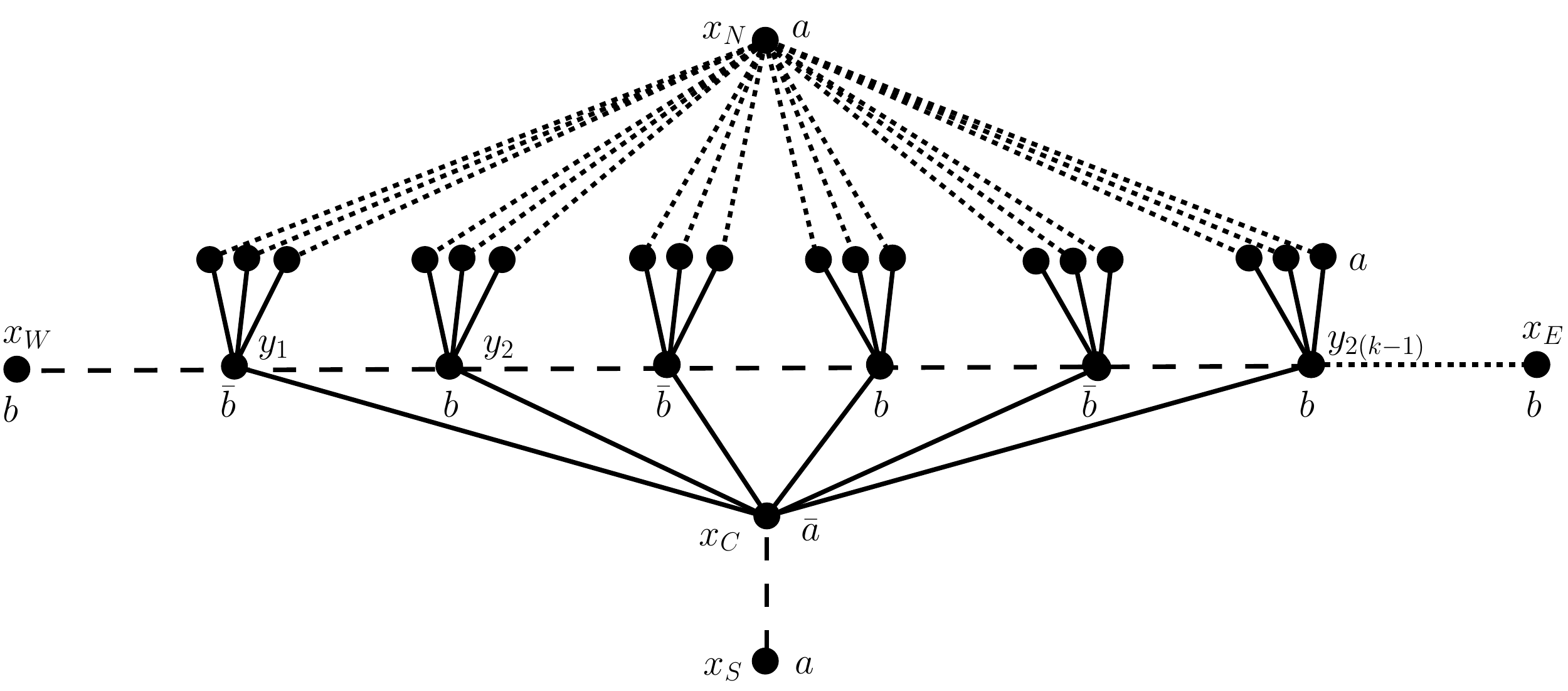}
\caption{Uncrosser gadget $U_k$ with $k=4$.}
\label{kk1}
\end{center}
\end{figure}

Consider an $MC(k(k-1))$-coloring $c$ of $U_k$ such that
$c(x_N)=a\in\{0,1\}$ and $c(x_W)=b\in\{0,1\}$ (note that possibly
$a=b$). In what follows, we write $\bar{a}$ instead of $1-a$ and
$\bar{b}$ instead of $1-b$. By the properties of the dotted and dashed
edges, $c(y_{2i+1})=\bar{b}$, $c(y_{2i})=c(x_E)=b$. Every vertex $y_i$
has $k-1$ neighbors that are linked to $x_N$ with dotted edges.  These
$k-1$ neighbors are thus colored $a$. The vertex $x_C$ is adjacent to
$k-1$ vertices colored $b$ and $k-1$ vertices colored $\bar{b}$.  In
particular, $x_C$ is adjacent to $k-1$ vertices colored $a$, and each
of them is adjacent to $k-1$ vertices colored $a$.  So $x_C$ cannot be
colored $a$, since it would create a monochromatic component of size
$k(k-1)+1$. Thus we have $c(x_C)=\bar{a}$ and $c(x_S)=a$.  This proves
property (1). To prove property (2), observe that the 2-coloring we
just considered contains only monochromatic components of size at most
$k$, regardless whether $a=b$ or not.

To construct $G$, we use copies of $K_{2,2k(k-1)-1}$ to transfer the
color of the primitive vertices to the vertices of the paths as we did
in the in the previous proof. We can draw the graph in the plane in
such a way that the edges of the paths do not cross any other edge
(for instance by drawing each path on the line of equation $x=0$). The
obtained graph is not necessarily planar, so we replace each crossing
of edges by a copy of $U_k$ in order to obtain a planar graph $G$.
\end{proof}

\noindent \textbf{Remark.}  We can modify the gadget $J_{y,z,t}$ in
the proof of Theorem~\ref{npc}(1), so that the same result holds for
graphs with arbitrarily large (but fixed) girth. Note that in this
case we lose the 2-degeneracy. The new gadget consists in the
bipartite double cover of a good expander (for instance, a Ramanujan
graph) having large girth and degree significantly larger that
$k$. The vertices $y$ and $z$ are any pair of (far apart) vertices on the same
size of the bipartition. Using the vertex expansion property, it can
be proven that such a graph admits no $MC(k)$-coloring other than the
proper 2-coloring (and therefore $y$ and $z$ are always colored the
same in such a coloring). We omit the details.

\medskip

\noindent \textbf{Related (recent) results.} After this paper was submitted, two manuscripts dealing with similar
topics appeared. In~\cite{CG15}, Chappell and Gimbel conjectured that for any
fixed surface $\Sigma$ there is a constant $k$ such that every graph
embeddable on $\Sigma$ can be 5-colored without monochromatic components
of size more than $k$. Note that our Theorem~\ref{th:g3col} proves this
conjecture in a strong sense. In~\cite{AUW15}, Axenovich, Ueckerdt,
and Weiner proved the following strong variant of our Theorem~\ref{th:girth6p}:
any planar graph of girth at least 6 has a 2-coloring such that each
monochromatic component is a path on at most 14 vertices.

\medskip

\noindent \textbf{Acknowledgement.} The authors would like to thank
Luke Postle for the interesting discussions, and an anonymous referee
for pointing out reference~\cite{JV05} and suggesting a simplification
in the proof of Theorem~\ref{th:g3island}.


\begin{thebibliography}{99}



\bibitem{ADOV03} N. Alon, G. Ding, B. Oporowski, and D. Vertigan,
  \emph{Partitioning into graphs with only small components},
  J. Combin. Theory Ser. B {\bf 87} (2003), 231--243.

\bibitem{AUW15} M. Axenovich, T. Ueckerdt,
and P. Weiner, \emph{Splitting Planar Graphs of Girth 6 into Two
  Linear Forests with Short Paths}, arXiv:1507.02815. 


\bibitem{BKY13} O.V. Borodin, A. V. Kostochka, and M. Yancey, \emph{On
  1-improper 2-coloring of sparse graphs}, Manuscript, 2013.

\bibitem{CK10} W.~Cushing, H.A.~Kierstead, \emph{Planar graphs are
  1-relaxed, 4-choosable}, European J. Combin. {\bf 31(5)} (2012), 1385--1397.

\bibitem{EJ13} L.~Esperet and G.~Joret, \emph{Coloring
  planar graphs with three colors and no large monochromatic
  components}, Combin. Prob. Comput. {\bf 23(4)} (2014), 551--570.

\bibitem{CG15} G. Chappell and J. Gimbel, \emph{On Subgraphs Without Large Components}, to appear in Mathematica Bohemica.

\bibitem{G59} H.~Gr\"{o}tzsch, {\em Ein Dreifarbensatz f\"ur dreikreisfreie
  Netze auf der Kugel}, Wiss. Z. Martin-Luther-Univ. Halle-Wittenberg
  Math.-Natur. Reihe {\bf 8} (1959) 109--120.

\bibitem{HST03} P. Haxell, T. Szab\'o, G. Tardos, \emph{Bounded size
  components--partitions and transversals}, J. Combin. Theory Ser. B
  {\bf 88} (2003), 281-- 297.


\bibitem{JV02} S. Jendrol' and H.-J. Voss, \emph{Light subgraphs of
  graphs embedded in 2-dimensional manifolds of Euler characteristic
  $\le 0$--a survey}, in: G. Hal\'asz, L. Lov\'asz, M. Simonovits,
  V.T. S\'os (Eds.), P. Erd\H os and his Mathematics II, Bolyai
  Society Mathematical Studies, vol. 11, Springer, Budapest, 2002,
  pp. 375--411.

\bibitem{JV05} S. Jendrol' and H.-J. Voss, \emph{Light subgraphs of
  order at most 3 in large maps of minimum degree 5 on compact
  2-manifolds}, European J. Combin. {\bf 26} (2005), 457--471.

  \bibitem{JV13} S. Jendrol' and H.-J. Voss, \emph{Light subgraphs of
    graphs embedded in the plane--A survey}, Discrete Math. {\bf 313(4)} (2013), 406--421.

\bibitem{KT12} K.~Kawarabayashi and C.~Thomassen, \emph{From the plane
  to higher surfaces}, J. Combin. Theory Ser. B {\bf 102} (2012),
  852--868.

\bibitem{KMRV97} J. Kleinberg, R. Motwani, P. Raghavan, and
  S. Venkatasubramanian, \emph{Storage management for evolving
    databases}, Proceedings of the 38th Annual IEEE Symposium on
  Foundations of Computer Science (FOCS 1997), 353--362.

  \bibitem{LMST08} N. Linial, J. Matou\v sek, O. Sheffet, and G. Tardos,
  \emph{Graph coloring with no large monochromatic components},
  Combin. Prob. Comput. {\bf 17(4)} (2008), 577--589.

  \bibitem{Lov73} L.~Lov\'asz, \emph{Coverings and colorings of
    hypergraphs}, Proceedings of the fourth south-eastern conference
    on combinatorics, graph theory, and computing, Boca Raton,
    Florida, 3--12, 1973.


\bibitem{MoTh} B.~Mohar and C.~Thomassen, \emph{Graphs on Surfaces}.
  Johns Hopkins University Press, Baltimore, 2001.

\bibitem{T94} C.~Thomassen, {\em Every planar graph is 5-choosable},
J.~Combin.\ Theory Ser.~B~{\bf 62} (1994) 180--181.

\bibitem{V95} M.~Voigt, {\em A not 3-choosable planar graph without
  3-cycles}, Discrete Math.~{\bf 146} (1995)
325--328.

\end{thebibliography}
\end{document}